\documentclass[12pt]{article}
\usepackage{graphicx}
\usepackage{indentfirst}
\usepackage{url}
\usepackage{amsmath,amsfonts,amsthm,amssymb,color}
\usepackage{cite}
\usepackage[pagebackref=true,colorlinks]{hyperref}

\usepackage{enumitem}
\setlist[enumerate]{itemsep=0mm}
\usepackage{algorithm}
\usepackage{algpseudocode}

\hypersetup{
	colorlinks=true,
	linkcolor=blue,
}
\newtheorem*{theorem*}{Theorem}
\newtheorem{theorem}{Theorem}[section]

\newtheorem{lemma}[theorem]{Lemma}

\theoremstyle{definition}
\newtheorem{definition}[theorem]{Definition}
\newtheorem{remark}[theorem]{Remark}

\begin{document}
\title{The Largest Pure Partial Planes of Order 6 Have Size 25}
\author{Yibo Gao}
\date{}
\maketitle
\begin{abstract}
In this paper, we prove that the largest pure partial plane of order 6 has size 25. At the same time, we classify all pure partial planes of order 6 and size 25 up to isomorphism. Our major approach is computer search. The search space is very large so we use combinatorial arguments to rule out some of the cases. For the remaining cases, we subdivide each search by phases and use multiple checks to reduce search space via symmetry.
\end{abstract}
\section{Introduction}
The problem of the existence of finite projective planes of certain orders has been attracting mathematicians' interest for hundreds of years. However, the problem still remains widely open. People know that finite projective planes of order equal to prime powers exist and no finite projective planes of order not equal to a prime power have been found. Therefore, some mathematicians conjectured that finite projective planes can only have prime power orders.

There was some progress made in the past. In 1938, Bose \cite{bose1938application} proved that there is no projective plane of order 6 by relating the existence of a finite projective plane to the existence of a \textit{hyper-Graeco-Latin square}, which is known as \textit{orthogonal Latin squares} in modern terminology. In 1949, Bruck and Ryser \cite{bruck1949nonexistence} proved that if the order $n$ is congruent to 1 or 2 modulo 4 and $n$ cannot be represented as the sum of two perfect squares, then there does not exist any finite projective planes of order $n$. This result is known as Bruck-Ryser theorem. By this famous combinatorial theorem, infinitely many cases are solved but still infinitely many cases are left.

Due to Bose's result and the Bruck-Ryser theorem, the smallest unsolved case is order $n=10$. After some progress using binary codes \cite{assmus1970possibility}, Lam, Thiel and Swierczp \cite{lam1989non} proved the nonexistence of finite projective planes of order 10 with the help of super computers and a total of 2 to 3 years of running time.

Aside from finding more finite projective planes or proving their nonexistence, there is one more attractive question. We already know that there is no finite projective planes of order 6, but how close can we come to constructing such a plane? In particular, what is the largest pure partial plane (see Definition \ref{Def:ppp}) of order 6 we can construct? In \cite{hering2007partial}, a pure partial plane of order 6 related to icosahedron is constructed and in \cite{prince2009pure}, two pure partial planes of order 6 with 25 lines that extend the dual of the point-line incidence structure of $PG(3,2)$, the three-dimensional projective geometry, are constructed. In \cite{mccarthy1976approximations}, McCarthy et al. proved that there are no pure partial plane of order 6 and size 29 with very long combinatorial arguments. However, the exact maximum has not been given.

In this paper, we prove that the maximum size of pure partial planes of order 6 is 25. Or in other words, a pure partial plane of order 6 contains at most 25 lines. 

\begin{theorem}\label{thm:main}
The maximum size of a pure partial plane of order 6 is 25. Furthermore, all pure partial planes of size 25 are listed in Appendix \ref{Sec:allPPP}.
\end{theorem}

To do this, we use computer search combined with standard combinatorial arguments. In Section \ref{Sec:Prelim}, we define the notion of saturated pure partial planes and introduce some related and useful lemmas. In Section \ref{Sec:Alg}, we give our algorithmic model for computer search and in Section \ref{Sec:Result}, we present our results. We also attach our code for readers to verify. A list of codes that we provide is shown in Appendix \ref{Sec:files}. Finally, the proof of our main theorem comes in Section \ref{Sec:Main}. All the pure partial planes of order 6 and size 25 are provided in Appendix \ref{Sec:allPPP}.
\section{Preliminaries}\label{Sec:Prelim}
In this paper, a ``point" means an element in a universe and a ``line" means a subset of this universe, or equivalently, a set of ``points". We will consider ``points" and ``lines" only in a set-theoretic view and won't discuss any finite geometry here.
\begin{definition}
A \textit{finite projective plane (FPP) of order} $n$, or a \textit{projective plane of order} $n$, is a collection of $n^2+n+1$ points and $n^2+n+1$ lines, such that
\begin{enumerate}
\item[(1)] every line contains $n+1$ points;
\item[(2)] every point is on $n+1$ lines;
\item[(3)] every two distinct lines intersect at exactly one point;
\item[(4)] every any two distinct points lies on exactly one line.
\end{enumerate}
\end{definition}

\begin{definition}\label{Def:ppp}
A \textit{pure partial plane (PPP) of order} $n$ \textit{and size} $s$ is a collection of $n^2+n+1$ points and $s$ lines, such that
\begin{enumerate}
\item[(1)] every line contains $n+1$ points;
\item[(2)] every two distinct lines intersect at exactly one point;
\end{enumerate}
\end{definition}
In Definition \ref{Def:ppp}, we say that there are $n^2+n+1$ points. Just for clarity, there can be at most $n^2+n+1$ points since some points may not appear in any of the lines.
\begin{definition}
We say that a pure partial plane is \textit{saturated} if no lines can be added to it such that it still remains a pure partial plane. We use the abbreviation SPPP for \textit{saturated pure partial plane}.
\end{definition}
\begin{definition}
Two pure partial planes are isomorphic if there exists a bijection of their points and a bijection of their lines such that the point-in-line relation is equivalent under these two bijections.
\end{definition}
We are only interested in (saturated) pure partial planes up to isomorphism. And for the rest of the paper, we will consider two (saturated) pure partial planes to be the same if they are isomorphic. In other words, we only care about isomorphism classes.

For convenience, we make the following definition.
\begin{definition}
We say that two lines are \textit{compatible} if they intersect at exactly one point and that two sets of lines are \textit{compatible} if every line from one set is compatible with every line from the other set.
\end{definition}

It is immediate that a finite projective plane is always a saturated pure partial plane of the same order, and is a largest one, in terms of the size.

From now on, we will always use $n$ for the order and $s$ for the size.

For convenience, we label all points as $0,1,\ldots,n^2+n$ and represent straightforwardly a line as a set of cardinality $n+1$, e.g., $\{0,1,2,3,4,5,6\}$.

\begin{lemma}\label{Lem:NoN}
For a saturated pure partial plane of order $n$, no points appear in exactly $n$ lines.
\end{lemma}

\begin{proof}
We use proof by induction. Suppose that there is a SPPP with point 0 appearing in $n$ lines. Assume that these $n$ lines are $\{0,1,\ldots,n\},\ \{0,n+1,\ldots,2n\},\ \ldots,\ \{0,n^2-n+1,\ldots,n^2\}$. For any other line $L$, it does not contain point $0$, so by definition, it must intersect $\{in+1,in+2,\ldots,in+n\}$ at exactly one point, for $i=0,1,\ldots,n-1$. Since $L$ contains $n+1$ points in total, we know that $L$ must intersect with $\{n^2+1,n^2+n,\ldots,n^2+n\}$ at exactly one point, too. Line $L$ does not contain 0 so $L$ intersects with $\{0,n^2+1,\ldots,n^2+n\}$ at exactly one point. It is then obvious that we can add a new line $\{0,n^2+1,\ldots,n^2+n\}$ to this collection, contradicting the property of saturation.
\end{proof}

\begin{lemma}\label{Lem:SumA}
For a pure partial plane of order $n$ and size $s$, suppose that there are $a_k$ points that appear in $k$ lines, for $k=0,1,\dots$. Then
$$\sum_{k}ka_k=(n+1)s;\qquad\sum_{k}k^2a_k=s^2+ns.$$
\end{lemma}
\begin{proof}
For each line $i$ in this pure partial plane, with $i=1,\ldots,s$, associate a vector $L_i\in\{0,1\}^{n^2+n+1}$ with it, such that $L_{i,j}=1$ if point $j$ appears in line $i$ and equals 0 otherwise. Let $v=L_1+\cdots+L_s$. The entries of $v$ are just a permutation of $m_0,m_1,\ldots,m_{n^2+n}$ so the $L^1$ norm of $v$ is $\sum ka_k$ and at the same time, it is the sum of the $L^1$ norms of $L_i$'s, giving us $(n+1)s$.

At the same time, $v\cdot v=\sum k^2a_k$. By definition, $L_i\cdot L_j=1$ if $i\neq j$ and $L_i\cdot L_j=n+1$ if $i=j$ so $v\cdot v=(n+1)s+(s^2-s)=s^2+ns$, as desired.
\end{proof}
Lemma \ref{Lem:SumA} is a simple but useful lemma that has appeared in other forms in previous works. For example, \cite{mccarthy1976approximations} mentions essentially the same thing in Section 3 but in a different format.

\begin{lemma}\label{Lem:SumC}
Suppose that $\{i_1,i_2,\ldots,i_{n+1}\}$ is a line in a pure partial plane of order $n$ and size $s$, and suppose that point $i_k$ appears $c_{i_k}$ times. Then $$c_{i_1}+c_{i_2}+\cdots+c_{i_{n+1}}=s+n.$$
\end{lemma}
\begin{proof}
All the lines in this pure partial plane are either the line $\{i_1,\ldots,i_{n+1}\}$ or contain exactly one of $i_1,\ldots,i_{n+1}$. Since $i_k$ appears $c_{i_k}$ times, there are $c_{i_k}-1$ lines that contain $i_k$ but not $i_j$ for all $j\neq k$ with $1\leq j\leq n+1$. Therefore, we have $(c_{i_1}-1)+(c_{i_2}-1)+\cdots+(c_{i_{n+1}}-1)+1=s$ and thus
$$c_{i_1}+c_{i_2}+\cdots+c_{i_{n+1}}=s+n.$$
\end{proof}

\begin{theorem}\label{Thm:even}
For any saturated pure partial plane of even order, there exists a point that appear in at least 3 lines.
\end{theorem}
\begin{proof}
Assume the opposite that there exists a SPPP such that the order $n$ is even and every point appears in at most 2 lines.

By Lemma \ref{Lem:SumA} and following its notation, we have $a_1+2a_2=(n+1)s$ and $a_1+4a_2=s^2+ns$. It gives $a_1=s(n+2-s)$ and $a_2=\frac{s^2-s}{2}={s\choose2}$.
Since $a_1\geq0$, $s\leq n+2$. Notice that the value of $a_2$ is actually obvious because no three lines intersect at one point so the intersection points of different pairs of lines are different.

When $n=2$, we can first assume that our first two lines are $\{0,1,2\}$ and $\{0,3,4\}$. Forcing point 0 to appear in only 2 lines, we can determine the next two lines to be $\{1,3,5\}$ and $\{2,4,6\}$, which are unique up to permutation. Now, the size of this pure partial plane is already 4, which equals $n+2$. However, it is not saturated since the line $\{0,5,6\}$ is compatible with it. This yields a contradiction.

So we assume that $n\geq4$. Then $$a_1+a_2=-\frac{1}{2}s^2+\frac{2n+3}{2}s=\frac{n^2+3n+2}{2}-(s-n-1)(s-n-2)\leq\frac{n^2+3n+2}{2}\leq n^2$$ as $n\geq4$. So $a_0\geq n+1$, meaning that we have plenty of points to use.

Suppose that points $0,1,\ldots,a_2-1$ appear two times and points $n^2,n^2+1,\ldots,n^2+n$ do not appear. Label the lines as $1,2,\ldots,s$.

If $s$ is even, say point 0 appears in lines $1,2$, point 1 appears in lines $3,4$, $\ldots$, point $\frac{s}{2}-1$ appears in lines $s-1,s$, then we can add a new line $\{0,1,\ldots,\frac{s}{2}-1,n^2,n^2+1,\ldots,n^2+n-\frac{s}{2}\}$. It is easy to see that this new line intersects with previous lines at exactly one point, as it intersects lines $2k-1,2k$ at point $k-1$ only, for $k=1,\ldots,\frac{s}{2}$.

If $s$ is odd, then as $n$ is even, $a_1=s(n+2-s)\neq0$. So we can assume that point $n^2-1$ appears exactly one time and line $s$ contains it. Further, since all pairs of lines intersect at some point, we can assume that point 0 appears in lines $1,2$, point 1 appears in lines $3,4$, $\ldots$, point $\frac{s-3}{2}$ appears in lines $s-2,s-1$. Similarly, a new line, $\{0,1,\ldots,\frac{s-3}{2},n^2-1,n^2,\ldots,n^2+n-\frac{s+1}{2}\}$ can be added.

Therefore, this pure partial plane cannot be saturated.
\end{proof}
\begin{theorem}\label{Thm:odd}
For each odd number $n\geq3$, there exists a saturated pure partial plane of order $n$ such that no points appear in more than 2 lines.
\end{theorem}
\begin{proof}
The construction is straightforward. Draw $n+2$ lines in $\mathbb{R}^2$ such that no two are parallel and no three are concurrent. This gives us ${n+2\choose 2}<n^2+n+1$ intersection points and $n+2$ lines, each passing through $n+1$ points. Clearly it is a pure partial plane. If it is not saturated, then we should be able to find a subset of these intersection points, as well as some previously unused points, such that each previously existing line passes through exactly one of them. However, each intersection points appear in exactly 2 lines and each previously unused points appear in exactly 0 lines, while there are $n+2$ previously existing lines. Because $n+2$ is odd, such a set cannot be found. Therefore, this construction indeed provides a SPPP as desired.

An example is given in Figure \ref{Fig:sppp3}.
\end{proof}
\begin{figure}[h!]
\centering
\includegraphics[scale=0.4]{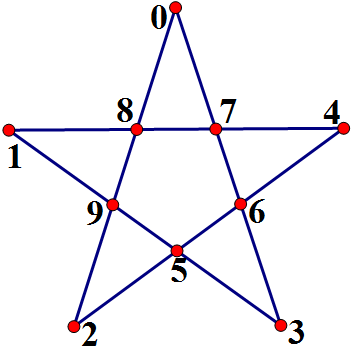}
\caption{a SPPP of order 3 and size 5}
\label{Fig:sppp3}
\end{figure}

Theorem \ref{Thm:even} and Theorem \ref{Thm:odd} are not related to our main theorem. It is still good to have them in the sense that we want to understand the notion of ``saturation" better.

The problem of testing isomorphism between pure partial planes can be reduced to the problem of graph isomorphism.
\begin{definition}\label{Def:PLAG}
For a pure partial plane of order $n$ and size $s$, define its \textit{point-line-adjacency graph} to be an undirected simple bipartite graph with $n^2+n+1+s$ vertices, representing all points and lines where a vertex representing a line is connected to a vertex representing a point if and only if the line contains the point.
\end{definition}
Notice that in the definition, we allow some vertices to have degree 0, although this detail is negligible.
\begin{theorem}\label{Thm:GraphIso}
Two pure partial planes of the same order that are not finite projective planes are isomorphic if and only if their point-line-adjacency graphs are isomorphic.
\end{theorem}
\begin{proof}
One direction is clear. If two pure partial planes are isomorphic, then their point-line-adjacency graphs are isomorphic. The bijection between points and the corresponding bijection between lines give a bijection between vertices in the graphs.

So now let $A$ and $B$ be two pure partial planes of the same order with point-line-adjacency graphs $G_A$ and $G_B$, respectively, and assume that $G_A$ is isomorphic to $G_B$. Clearly, $G_A$ and $G_B$ have the same number of vertices so we know that $A$ and $B$ have the same size $s$. If $s=0$, the theorem is correct. Now we assume $s>0$. Then, $G_A$ and $G_B$ must have the same number of vertices of degree 0, so the number of points appearing in $A$ is the same as the number of points appearing in $B$. For the other vertices of degree at least 1, $G_A$ and $G_B$ are bipartite since there are no edges between points and no edges between lines. By definition, since two lines intersect at exact one point, there must be a length 2 path from any vertex representing a line to another vertex representing a line. So all vertices representing lines are in one connected components, and therefore, this connected component contains all vertices representing points that appear at least once. So, the number of connected components in $G_A$ and $G_B$, ignoring the degree 0 vertices, is 1.

Now, the isomorphism between $G_A$ and $G_B$ will first pair up their degree 0 vertices. And then, it will pair up one part (of the bipartite graph) of $G_A$ to one part of $G_B$. We only need to make sure that vertices representing points in $G_A$ are not paired up with vertices representing lines in $G_B$. For that to happen, all points in $G_A$ that appear will have to appear $n+1$ times since all lines in $G_B$ contain $n+1$ points. The number of points that appear is $(n+1)\cdot s/(n+1)=s$.

Suppose points that appear are labeled as $1,\ldots,s$. For $i=1,\ldots,s$ let $L_i\in\{0,1\}^s$ be such that $L_{i,j}=1$ if line $i$ contains point $j$ and 0 otherwise. Then the dot product $L_i\cdot L_k=1$ if $i\neq k$ and $L_i\cdot L_i=n+1$. Let us compute $x=(L_1+\cdots+L_s)\cdot(L_1+\cdots+L_s)$. First, $L_1+\cdots+L_s=(n+1,n+1,\ldots,n+1)$ so $x=(n+1)^2\cdot s$. At the same time, $x=s\cdot(n+1)+s(s-1)$. Comparing these two, since $s>0$, we easily get $s=n^2+n+1$, which means that $A$ is a finite projective plane, contradicting our assumption.
\end{proof}

\begin{remark}
Theorem \ref{Thm:GraphIso} fails when we consider two finite projective planes that are not self-dual. Examples include Hall planes \cite{hall1943projective}.
\end{remark}

By Theorem \ref{Thm:GraphIso}, we are able to transform the problem of testing isomorphism between saturated pure partial planes into graph isomorphism. Therefore, we can then use the fastest online code for graph isomorphism, \textbf{nauty} and \textbf{Traces} \cite{McKay201494}, to do so.

\section{Programming Model}\label{Sec:Alg}

From now on, we will use computer search to find saturated pure partial planes. In this section, we will present a programming model for searching. Our model is highly adjustable with many conditions under specified. In the next section, we will go into details about specific cases and will specify the conditions that are unclear for now.

We start with a certain pure partial plane which we call a \textit{starting configuration}. By brute force, we then generate a list of lines for us to choose from, which we call a \textit{starting list}, that are compatible with this starting configuration. Using the starting configuration and this list of lines, we do a depth first search, adding line by line to this starting configuration from the list and removing incompatible lines from the list until the list becomes empty. Whenever we get a saturated pure partial plane (the corresponding list of compatible lines is empty), we check if it is isomorphic to any of the saturated pure partial planes we already have by Theorem \ref{Thm:GraphIso} and \textbf{nauty} and \textbf{Traces} \cite{McKay201494}. If not, we record this saturated pure partial plane.

Here is the basic algorithmic model for depth first search (DFS). In the following diagram, there are steps that are not specified since we use different implementations for different purposes, including Step \ref{Stp:End}, Step \ref{Stp:Saturated} and Step \ref{Stp:Check}. We will also explain them below.

\begin{algorithm}
\caption{Depth First Search for Saturated Pure Partial Planes}\label{Alg:dfs}
\begin{algorithmic}[1]
\Procedure{Depth-First-Search}{$ppp_0,rl_0$}\Comment{$ppp_0$ is a pure partial plane, $rl_0$ is a list of lines}
\If{$ppp_0,rl_0$ satisfy certain terminating properties}\label{Stp:End}
    \If{$ppp_0$ satisfies certain properties and is not isomorphic to all PPPs recorded already}\label{Stp:Saturated}
        \State Record $ppp_0$ globally
    \EndIf
\Else
    \For{each line $L$ in $rl_0$}
        \State $ppp_1\leftarrow ppp_0+L$\Comment{a new pure partial plane}
        \State Construct $rl_1$ from $rl_0$ by selecting the lines that intersect with $L$ at exactly one point
        \If{$ppp_1$ passes all the checks}\label{Stp:Check}
            \State \textsc{Depth-First-Search}($ppp_1,rl_1$)
        \EndIf
    \EndFor
\EndIf
\EndProcedure
\end{algorithmic}
\end{algorithm}

This paradigm is very straightforward and simple. However, usually the search space is very large so we need methods to cut down some symmetric cases beforehand.

In Step \ref{Stp:End}, the \textit{certain terminating properties} is usually implemented as checking if $rl_0$ is empty. We will assume so if not specified.

Step \ref{Stp:Check} is the main step in which we cut off symmetric cases. In this step, we will typically check the following properties of $ppp_1$:
\begin{itemize}
\itemsep0em
\item[1.] If point $i$ has appeared in this pure partial plane, then point $i-1$ must also appear in this pure partial plane, for $i=1,2,\ldots,n^2+n$.
\item[2.] The line $L$ just added must be lexicographically greater than the last line in $ppp_0$, assuming all points in each line are sorted.
\item[3.] The number of times that certain points appear should not exceed certain values. These parameters will be specified in Section \ref{Sec:Result} where we are using this algorithm.
\end{itemize}

Check 1 above in Step \ref{Stp:Check} is not always useful. When $n=6$, the cases we are dealing with usually have one point that appears 7 times, meaning that in the starting configuration, all points have already appeared.

Check 2 above in Step \ref{Stp:Check} can also be implemented in the way that in our DFS step after adding the new line to the pure partial plane, we discard all lines from the list of compatible lines that are lexicographically greater than this new line.

Check 3 above in Step \ref{Stp:Check} is the most important one. Typically we will divide the whole case by assuming the number of appearance of certain points. Here, we check that if the number of appearance of such points in $ppp_1$ has already exceeded our assumption.

Requiring some properties of $ppp_0$ in Step \ref{Stp:Saturated} usually helps us reduce the number of isomorphism testing. For example, if in our starting configuration, points $3,4,5,6$ are symmetric, then we can require that in $ppp_0$ the number of times that $3,4,5,6$ appear forms a non-decreasing sequence. In this way, some isomorphic cases will be quickly discarded.
\section{Search Results}\label{Sec:Result}
Our goal is to prove that all possible pure partial planes of order 6 have size at most 25 and to give all pure partial planes of order 6 and size 25. Following our previous notations, let $a_i$ be the number of points that appear exactly $i$ times. We will only consider saturated pure partial planes, in order to use Lemma \ref{Lem:NoN} and get that $a_6=0$, meaning that no points can appear exactly 6 times. Intuitively, if we want our SPPPs to have large sizes, we need to have the points appear as many times as possible. So as an overview, we will search for the cases where $a_7\geq2$ and also the cases where $a_5$ is sufficiently large. In the next section (Section \ref{Sec:Main}), we will give a proof showing that all possible SPPPs with size at least 25 are already covered in our search. And at that point, it will be clear why we discuss these cases.

In this section, we will consider five cases specified in each subsection. For each of them, we will use the algorithm given in Section \ref{Sec:Alg} in multiple phases. In each phase, the inputs are some pure partial planes regarded as starting configurations and the outputs are some bigger pure partial planes, that will be used as starting configurations for the next phase. Intuitively, using multiple phases instead of one will reduce search time since some symmetric cases can be cut off when they have not grown very big. Essentially, searching for pure partial planes in multiple phases is like doing breadth first search. Since we do isomorphism testings after each phase, the idea of combining breadth first search into the depth first search backbone can speed up the search. However, we want the number of phases to be small because breadth first search may consume too much space. For convenience, we will assume that $\{0,1,2,3,4,5,6\}$ is the first line in our starting configuration (except the last case). Also, whenever we talk about a particular ``Step", we are referring to our algorithmic model in Section \ref{Sec:Main}. 

We provide a list of programs for readers to verify (Appendix \ref{Sec:allPPP}).

\subsection{At least 3 points appear 7 times}
First, 0 appears 7 times. Assume that these 7 lines are $\{0,6k+1,6k+2,\ldots,6k+6\}$ where $k=0,1,\ldots,6$. At this stage, all other points are equivalent under symmetric group so we can safely assume that 1 appears 7 times, too. Let the next 6 lines be $\{1,k+7,k+13,k+19,k+25,k+31,k+37\}$ where $k=0,1,\ldots,5$. It is also clear that the choice of these 6 lines are unique.

Now that we have 13 lines in our starting configuration, we need to divide this case. For the third point that appears 7 times, it may be a point that appear in the same line with both 0 and 1, i.e. $2,3,4,5,6$ or other points. We divide this case into two subcases where 2 appears 7 times and where 7 appears 7 times. Notice that in both cases, we can add one more line $\{2,7,14,21,28,35,42\}$ into the collection using symmetry.
\subsubsection{Point 2 appears 7 times}\label{Res:2x7}

\noindent\textbf{Phase 1}

We use our program with stating configuration being this 14-line pure partial plane, the starting list being all lines that start with point 2 and are compatible with the starting configuration. In Step \ref{Stp:End} (described in Section \ref{Sec:Alg}), we simply require that point 2 appears 7 times or equivalently, the size of $ppp_0$ is 19. In Step \ref{Stp:Saturated}, we do nothing and in Step \ref{Stp:Check}, we only do check 2. which checks the lexicographical order.

Running the program gives us a total of 12 nonisomorphic pure partial planes of size 19, where 0,1,2 appear 7 times. These starting configurations are shown in file ``case1-1-phase1.txt".

\

\noindent\textbf{Phase 2}

Then we treat these 12 pure partial planes as starting configurations and run our program again, with the starting list being all lines that are compatible with the starting configuration. The search space is pretty small in this case so we do not actually need a lot of checks. The only check we implemented here is the check of lexicographical order in Step \ref{Stp:Check}. In Step \ref{Stp:End}, we require the list of lines $rl_0$ to be empty.

These 12 starting configurations provide 36 nonisomorphic saturated pure partial planes. The results are shown in file ``case1-1-phase2(SPPP).txt". Among these results, the maximum size is 25 and there are 3 SPPPs that achieve 25.
\subsubsection{Point 7 appears 7 times}
\noindent\textbf{Phase 1}

We use our program with starting configuration begin the 14-line pure partial plane described above, the starting list being all lines with point 7 and compatible with the starting configuration. Similarly, in Step \ref{Stp:End}, we require that point 7 must appear exactly 7 times, or equivalently, the size of $ppp_0$ is 18. And in Step \ref{Stp:Check}, we only do check 2. which checks the lexicographical order.

The program produces 2 nonisomorphic pure partial planes of size 18 where 0,1,7 appear 7 times. These starting configurations are shown in file ``case1-2-phase1.txt".

\

\noindent\textbf{Phase 2}

Then we use these 2 pure partial planes as starting configurations to get saturated pure partial planes using our program. In Step \ref{Stp:Check}, we require that points 2, 3, 4, 5, 6, 8, 9, 10, 11, 12, 13, 19, 25, 31, 37 can appear at most 5 times. Otherwise, if one of them appears at least 6 times, by Lemma \ref{Lem:NoN}, it must appear 7 times in the corresponding saturated pure partial planes and we are then back to the previous case where 0,1,2 appear 7 times.

We find that there are 30 SPPPs while none of these can achieve size 25. The results are shown in ``case1-2-phase2(SPPP).txt".
\subsection{Exactly 2 points appear 7 times}\label{Res:01x7}
In this case, we have only one phase. We assume that 0 and 1 appear 7 times and thus have the unique 13-lines starting configuration: $\{0,6k+1,6k+2,\ldots,6k+6\}$ where $k=0,1,\ldots,6$ and $\{1,k+7,k+13,k+19,k+25,k+31,k+37\}$ where $k=0,1,\ldots,5$. Actually, it can be easily seen that if we can add one more line $\{2,7,14,21,28,35,42\}$ while still keeping the uniqueness.

We use the program for this 14-line starting configuration and with the starting list being the list of all possible lines that are compatible with the starting configuration. In Step \ref{Stp:Saturated}, we require that $c_3\geq c_4\geq c_5\geq c_6$, where $c_i$ is the number of times that $i$ appears. This requirement is valid because in our starting configuration, points 3,4,5,6 are equivalent under the symmetric group. In Step \ref{Stp:Check}, we check the lexicographical order as usual and we also require that all points except 0,1 must appear at most 5 times, by Lemma \ref{Lem:NoN}.

Eventually, we get 2166 SPPPs while the maximum size is 23. The results are shown in ``case2-phase1(SPPP).txt".
\subsection{Exactly one point, 0, appears 7 times; 1,2,3 appear 5 times and 4 appears at least 4 times}\label{Res:0x7,123x5}
The reason that we do not do the case where exactly one point appears 7 times is largeness of our search space. Therefore, we restrict our attention to the case that one point appears 7 times while a lot of points appear 5 times. As before, $\{0,1,2,3,4,5,6\}$ is a line in our starting configuration. We have only two phases while the second phase has little work to do.

\

\noindent\textbf{Phase 1}

Let us determine the starting configuration. The first 7 lines are $\{0,6k+1,6k+2,\ldots,6k+6\}$ where $k=0,1,\ldots,6$ and the next 4 lines are $\{1,k+7,k+13,k+19,k+25,k+31,k+37\}$ where $k=0,1,2,3$. The next line, containing one of $2,3$ can also be uniquely added, which we assume to be $\{2,7,14,21,28,35,41\}$.

We use our program for this 12-line starting configuration with the starting list being the list of all possible lines that start with one of $2,3,4$ and are compatible with the starting configuration. In Step \ref{Stp:End}, we no longer require $rl_0$ to be empty; instead, we check that if $c_2=c_3=5$ and $c_4=4$, where $c_i$ is the number of times that $i$ appears in $ppp_0$. Equivalently, this is to say that the size of $ppp_0$ is 22 by Lemma \ref{Lem:SumC} used on the first line. In Step \ref{Stp:Check}, we check the lexicographical order as usual and also make sure that no point except 1 can appear more than 5 times.
The result from the program is 26 pure partial planes of size 22, presented in ``case4-phase1.txt".

\

\noindent\textbf{Phase 2}

The second phase is simply extending these pure partial planes to saturation. It turns out that some of them are already saturated and the others can be made saturated by appending one line. We get a 23 pure partial planes of size 22 or 23 in total, presented in ``case4-phase2(SPPP).txt".
\subsection{Points 0,1,2,3,4 appear 5 times each}\label{Res:01234x5}
Recall that we require $\{0,1,2,3,4,5,6\}$ to be our first line. Importantly, in this case, we do not actually require that no points appear 7 times, but rather, we require that point 0,1,2,3,4 appear exactly 5 times each. We will see from the results that actually no points can appear more than 5 times in all the saturated pure partial planes we get in the end.

\

\noindent\textbf{Phase 1}

As before, we can determine the first 9 lines uniquely. They are

\begin{tabular}{|l|ccccccc|}
\hline
line 1 & 0&1&2&3&4&5&6\\
\hline
line 2 & 0&7&8&9&10&11&12\\
\hline
line 3 & 0&13&14&15&16&17&18\\
\hline
line 4 & 0&19&20&21&22&23&24\\
\hline
line 5 & 0&25&26&27&28&29&30\\
\hline
line 6 & 1&7&13&19&25&31&32\\
\hline
line 7 & 1&8&14&20&26&33&34\\
\hline
line 8 & 1&9&15&21&27&35&36\\
\hline
line 9 & 1&10&16&22&28&37&38\\
\hline
\end{tabular}

\

In this phase, we add two lines that start at point 2 to this 9-line starting configuration. Namely, we use our program with the starting list being all possible lines that include point 2 and are compatible with the 9-line starting configuration shown above. And in Step \ref{Stp:End}, we require the size of $ppp_0$ to be 11.
In this way, we get 29 pure partial planes with size 11, used as starting configurations for our next phase. These starting configurations are presented in file ``case4-phase1.txt".

\

\noindent\textbf{Phase 2}

For each of the starting configuration with size 11 we just obtained, we use the program with the starting list being all possible lines that start at point 2,3,4 and are compatible with the starting configuration with 11 lines. In Step \ref{Stp:Check}, we make sure that 2,3,4 appear at most 5 times. And in Step \ref{Stp:End}, we require point 2,3,4 to appear exactly five times.
In this phase, we get 30 pure partial planes with size 21. They are shown in file ``case4-phase2.txt".

\

\noindent\textbf{Phase 3}

For each of the 21-line starting configurations, we run our program with the starting list being all possible lines that are compatible with the starting configuration. In Step \ref{Stp:Check}, we make sure that no point can appear more than 5 times and in Step \ref{Stp:End}, we require that the list $rl_0$ is empty, meaning that we require the pure partial plane to be saturated. Interestingly, 18 of these 21-line starting configurations are already saturated and the rest of them cannot be made saturated without letting one of point 0,1,2,3,4 appear 7 times. All possible saturated pure partial planes in this case are shown in file ``case4-phase3(SPPP).txt".

\subsection{Points 0,...,14 appear exactly 5 times and points 15,...,39 appear exactly 4 times}\label{Res:40}
In this case, we focus our attention to the situations where there are 15 points appearing 5 times, 25 points appearing 4 times and 3 points not appearing at all. Also, we require that in each line, three points are from $0,\ldots,14$ and four points are from $15,\ldots,39$. The use of this case will become clear in Section \ref{Sec:Main} where we give the main theorem. By a simple counting formula, we know that if such pure partial plane exists, it must have size 25.

For this case, we will need a different isomorphism testing function in order to differentiate between a point that appears 5 times and a point that appears 4 time. To do this, we simply add another vertex to our point-line-adjacency graphs (Definition \ref{Def:PLAG}), connect vertices $0,1,\ldots,14$ to it and use graph isomorphism testing for the new graphs. Notice that this different isomorphism testing function is used solely for this case.

The first 5 lines starting at 0 can be uniquely determined.

\begin{tabular}{|l|ccccccc|}
\hline
line 1 & 0&1&2&15&16&17&18\\
\hline
line 2 & 0&3&4&19&20&21&22\\
\hline
line 3 & 0&5&6&23&24&25&26\\
\hline
line 4 & 0&7&8&27&28&29&30\\
\hline
line 5 & 0&9&10&31&32&33&34\\
\hline
\end{tabular}

\

The next line starting at 1 must pair up with two points from $\{3,4,\ldots,14\}$. There are two possibilities: 1,3,5 or 1,3,11. Specifically, the lines are $\{1,3,5,27,31,35,36\}$ and $\{1,3,11,23,27,31,35\}$.

\

\noindent\textbf{Phase 1}

With these two possible starting configurations, we add 4 lines to them that start with 1. We run our program (separately for these two starting configurations) with all lines that start at 1, contain three points from $\{1,\ldots,14\}$ and four points from $\{15,\ldots,39\}$ and are compatible with the starting configuration. In Step \ref{Stp:End}, we check that $ppp_0$ has size 9. Here, we get a total of 13 pure partial planes of size 9, presented in ``case5-phase1.txt".

\

\noindent\textbf{Phase 2}

This phase exists because we want to save some running time. We add just 1 compatible line that starts at 2 to the starting configurations. Then we get a total of 620 pure partial planes of size 10, shown in ``case5-phase2.txt".

\

\noindent\textbf{Phase 3}

We run our program with all lines that are compatible with the starting configuration, contain three points from $\{2,\ldots,14\}$ and four points from $\{15,\ldots,39\}$. In Step \ref{Stp:Check}, we make sure that points $2,\ldots,14$ never appear more than 5 times and points $15,\ldots,39$ never appear more than 4 times. In Step \ref{Stp:End}, we check that if our pure partial plane has size 25. Finally, we get a single pure partial plane of size 25, shown in ``case5-phase3(SPPP).txt". In fact, it must be saturated and we will explain this in next Section.

\subsection{Summary}\label{Sec:Summary}
For all these 5 cases described above, we find no pure partial planes of size 26 or greater. We find a total of 4 pure partial planes of size 25: three from Section \ref{Res:2x7} and one from Section \ref{Res:40}. We will list all of them in Appendix \ref{Sec:allPPP} for clarity.

\section{Main Theorem}\label{Sec:Main}
In this section, we restate our main theorem and finish the rest of the proof.
\begin{theorem*}[\ref{thm:main}]
The maximum size of a pure partial plane of order 6 is 25. Furthermore, all pure partial planes of size 25 are listed in Appendix \ref{Sec:allPPP}.
\end{theorem*}
\begin{proof}
Essentially, we want to show that there are no pure partial planes of size 25 or greater outside our search. To do this, we restrict our attention to saturated pure partial planes instead of pure partial planes in general because we want to use Lemma \ref{Lem:NoN}.

Assume that there exists a saturated pure partial plane $A$ of size $s$ with $s\geq25$. Specifically, assume that $A$ is a saturated pure partial plane that is not mentioned in our search in Section \ref{Sec:Result}. Define $a_i$ to be the number of points that appear $i$ times in $A$. Since we have already searched all possible cases for $a_7\geq2$ (Section \ref{Sec:Result}), now we assume that $a_7\leq1$. Use $c_i$ to denote the number of times that point $i$ appears in $A$. In other words, $c_i$ is the number of lines in $A$ that contain point $i$.

Lemma \ref{Lem:SumA} gives us the following useful equations, with $n=6$:
\begin{align*}
7a_7+5a_5+4a_4+3a_3+2a_2+a_1=&7s,\\
49a_7+25a_5+16a_4+9a_3+4a_2+a_1=&s^2+6s.
\end{align*}
Notice that according to Lemma \ref{Lem:NoN}, $a_6=0$ so we ignore this term.

In the above two equations, we subtract the second one by 5 times the first one, in order to get rid of $a_5$, which is potentially the largest term. We then divide this equation by 2. Together with the first equation, we have
\begin{equation}\label{Eqn:1}
7a_7+5a_5+4a_4+3a_3+2a_2+a_1=7s,
\end{equation}
\begin{equation}\label{Eqn:11}
2a_4+3a_3+3a_2+2a_1=\displaystyle{\frac{29s-s^2}{2}}+7a_7.
\end{equation}

\noindent\textbf{Case 1:} $a_7=0$ and $s\geq 26$.

For any line $\{i_1,\ldots,i_7\}$ of $A$, according to Lemma \ref{Lem:SumC}, we have $c_{i_1}+\cdots+c_{i_7}=s+6\geq32$. $a_7=0$ means that $c_{i_k}\leq5$ for $k=1,\ldots,7$. Also, according to the search result from Section \ref{Res:01234x5}, we have already covered the cases where there are at least 5 points that appear 5 times in a line. So we know that at most 4 of $c_{i_1},\ldots,c_{i_7}$ can be 5. Then $s+6\leq 5+5+5+5+4+4+4$ so $s\leq 26$. There is only one possibility now: $s=26$, 4 of $c_{i_1},\ldots,c_{i_7}$ equal 5 and the other 3 equal 4. In other words, $a_k>0$ only when $k=4,5$.

Equation \eqref{Eqn:1} and \eqref{Eqn:11} become $5a_5+4a_4=182$, $2a_4=39$. It clearly does not have integer solutions.

\

\noindent\textbf{Case 2:} $a_7=0$ and $s=25$.

For any line $\{i_1,\ldots,i_7\}$ of $A$, according to Lemma \ref{Lem:SumC}, we have $c_{i_1}+\cdots+c_{i_7}=s+6=31$.
Similarly as above, since we have already searched for cases where at least 5 points in this line appear a total of 5 times, there are these two possibilities left for $c_{i_1},\ldots,c_{i_7}$: $5,5,5,5,4,4,3$ and $5,5,5,4,4,4,4$. Thus, we must have that $a_1=a_2=0$.

Equation \eqref{Eqn:1} and \eqref{Eqn:11} become $5a_5+4a_4+3a_3=175$, $2a_4+3a_3=50$. The second one gives that $a_3$ is a multiple of $2$. We subtract the first equation by two times the second equation and get $5a_5-3a_3=75$ and it gives that $a_3$ is a multiple of 5. Thus, $a_3$ is a multiple of 10. Since a point that appears 3 times must be contained in 3 lines of the form 5,5,5,5,4,4,3, we then know that $3a_3\leq 25$. These arguments give us $a_3=0$. By solving the equations, we get $a_5=15$ and $a_4=25$. Further, every line has the form $5,5,5,4,4,4,4.$ This is exactly the case we considered in Section \ref{Res:40}.

\

\noindent\textbf{Case 3:} $a_7=1$.

We can assume that 0 appears 7 times and the lines that contain 0 are $\{0,6k+1,6k+2,\ldots,6k+6\}$ where $k=0,1,\ldots,6$. Therefore, we can see that all points have appeared at least once so $a_0=0$. We thus have an additional equation $a_7+a_5+a_4+\cdots+a_1=n^2+n+1=43$.

If $a_1\neq0$, without loss of generality, we assume that point 6 appears $1$ time, meaning $c_6=1$. According to Lemma \ref{Lem:SumC}, $c_0+c_1+\cdots+c_6=s+6\geq31$. So $c_1+c_2+c_3+c_4+c_5\geq23$ with $c_i\leq 5$ for $i=1,\ldots,5$. Therefore, either at least four of $c_1,c_2,c_3,c_4,c_5$ have value 5 or at least three of them have value 5 and a fourth one have value at least 4. These situations is covered in our computer search in Section \ref{Res:0x7,123x5}. So we then assume that $a_1=0$.

If $a_2\neq0$, assume that point 6 appears in two lines. At least one of these two lines won't contain point 0 since they intersect at point 6 already. Suppose that this line is $6,j_1,j_2,\ldots,j_6$. Then by Lemma \ref{Lem:SumC}, $c_6+c_{j_1}+\cdots+c_{j_6}\geq31$. Since $c_6=2$ and $c_{j_i}\leq5$ for $i=1,\ldots,6$, at least five of $c_{j_1},\ldots,c_{j_6}$ must be 5. This situation is covered in our computer search in Section \ref{Res:01234x5}. So we then assume that $a_2=0$.

Simplify Equations \eqref{Eqn:1}, \eqref{Eqn:11} and together with the new equation, we now have
\begin{equation}\label{Eqn:2}
5a_5+4a_4+3a_3=7s-7,
\end{equation}
\begin{equation}\label{Eqn:3}
2a_4+3a_3=\frac{29s-s^2}{2}+7,
\end{equation}
\begin{equation}\label{Eqn:4}
a_5+a_4+a_3=42.
\end{equation}

Manipulating the equations by $3\cdot\text{Equation }\eqref{Eqn:2}+2\cdot\text{Equation }\eqref{Eqn:3}-15\cdot\text{Equation }\eqref{Eqn:4}$, we get
$$a_4=50s-s^2-637=-12-(s-25)^2<0,$$
a clear contradiction.

Therefore, there are no saturated pure partial planes of order 6 and size at least 25 that are outside of our search.
\end{proof}

\section{Appendix}\label{Sec:Appendix}
\subsection{All Pure Partial Planes of Order 6 and Size 25}\label{Sec:allPPP}
Here is a list of all (saturated) pure partial planes of order 6 and size 25, up to isomorphism.

\

\noindent\{\{0, 1, 2, 3, 4, 5, 6\}, \{0, 7, 8, 9, 10, 11, 12\}, \{0, 13, 14, 15, 16, 17, 18\}, \{0, 19, 20, 21, 22, 23, 24\}, \{0, 25, 26, 27, 28, 29, 30\}, \{0, 31, 32, 33, 34, 35, 36\}, \{0, 37, 38, 39, 40, 41, 42\}, \{1, 7, 13, 19, 25, 31, 37\}, \{1, 8, 14, 20, 26, 32, 38\}, \{1, 9, 15, 21, 27, 33, 39\}, \{1, 10, 16, 22, 28, 34, 40\}, \{1, 11, 17, 23, 29, 35, 41\}, \{1, 12, 18, 24, 30, 36, 42\}, \{2, 7, 14, 21, 28, 35, 42\}, \{2, 8, 13, 22, 27, 36, 41\}, \{2, 9, 16, 23, 30, 31, 38\}, \{2, 10, 15, 24, 29, 32, 37\}, \{2, 11, 18, 19, 26, 34, 39\}, \{2, 12, 17, 20, 25, 33, 40\}, \{3, 7, 15, 20, 30, 34, 41\}, \{3, 9, 14, 19, 29, 36, 40\}, \{3, 10, 13, 23, 26, 33, 42\}, \{4, 7, 18, 22, 29, 33, 38\}, \{4, 11, 13, 21, 30, 32, 40\}, \{4, 12, 14, 23, 27, 34, 37\}\}.

\

\noindent\{\{0, 1, 2, 3, 4, 5, 6\}, \{0, 7, 8, 9, 10, 11, 12\}, \{0, 13, 14, 15, 16, 17, 18\}, \{0, 19, 20, 21, 22, 23, 24\}, \{0, 25, 26, 27, 28, 29, 30\}, \{0, 31, 32, 33, 34, 35, 36\}, \{0, 37, 38, 39, 40, 41, 42\}, \{1, 7, 13, 19, 25, 31, 37\}, \{1, 8, 14, 20, 26, 32, 38\}, \{1, 9, 15, 21, 27, 33, 39\}, \{1, 10, 16, 22, 28, 34, 40\}, \{1, 11, 17, 23, 29, 35, 41\}, \{1, 12, 18, 24, 30, 36, 42\}, \{2, 7, 14, 21, 28, 35, 42\}, \{2, 8, 13, 22, 27, 36, 41\}, \{2, 9, 16, 23, 30, 31, 38\}, \{2, 10, 15, 24, 29, 32, 37\}, \{2, 11, 18, 19, 26, 34, 39\}, \{2, 12, 17, 20, 25, 33, 40\}, \{3, 7, 15, 20, 30, 34, 41\}, \{3, 9, 14, 19, 29, 36, 40\}, \{3, 10, 13, 23, 26, 33, 42\}, \{4, 7, 18, 23, 27, 32, 40\}, \{4, 11, 14, 22, 30, 33, 37\}, \{4, 12, 13, 21, 29, 34, 38\}\}.

\

\noindent\{\{0, 1, 2, 3, 4, 5, 6\}, \{0, 7, 8, 9, 10, 11, 12\}, \{0, 13, 14, 15, 16, 17, 18\}, \{0, 19, 20, 21, 22, 23, 24\}, \{0, 25, 26, 27, 28, 29, 30\}, \{0, 31, 32, 33, 34, 35, 36\}, \{0, 37, 38, 39, 40, 41, 42\}, \{1, 7, 13, 19, 25, 31, 37\}, \{1, 8, 14, 20, 26, 32, 38\}, \{1, 9, 15, 21, 27, 33, 39\}, \{1, 10, 16, 22, 28, 34, 40\}, \{1, 11, 17, 23, 29, 35, 41\}, \{1, 12, 18, 24, 30, 36, 42\}, \{2, 7, 14, 21, 28, 35, 42\}, \{2, 8, 13, 22, 27, 36, 41\}, \{2, 9, 17, 19, 30, 32, 40\}, \{2, 10, 18, 23, 25, 33, 38\}, \{2, 11, 16, 24, 26, 31, 39\}, \{2, 12, 15, 20, 29, 34, 37\}, \{3, 7, 15, 23, 26, 36, 40\}, \{3, 8, 17, 24, 28, 33, 37\}, \{3, 11, 13, 21, 30, 34, 38\}, \{4, 7, 16, 20, 30, 33, 41\}, \{4, 8, 18, 21, 29, 31, 40\}, \{4, 12, 13, 23, 28, 32, 39\}\}.

\

\noindent\{\{0, 1, 2, 15, 16, 17, 18\}, \{0, 3, 4, 19, 20, 21, 22\}, \{0, 5, 6, 23, 24, 25, 26\}, \{0, 7, 8, 27, 28, 29, 30\}, \{0, 9, 10, 31, 32, 33, 34\}, \{1, 3, 5, 27, 31, 35, 36\}, \{1, 4, 6, 28, 32, 37, 38\}, \{1, 11, 12, 19, 23, 29, 33\}, \{1, 13, 14, 20, 24, 30, 34\}, \{2, 7, 9, 19, 24, 35, 37\}, \{2, 8, 10, 20, 23, 36, 38\}, \{2, 11, 13, 21, 25, 27, 32\}, \{2, 12, 14, 22, 26, 28, 31\}, \{3, 7, 11, 15, 26, 34, 38\}, \{3, 8, 14, 16, 25, 33, 37\}, \{3, 9, 13, 17, 23, 28, 39\}, \{4, 8, 11, 18, 24, 31, 39\}, \{4, 9, 12, 15, 25, 30, 36\}, \{4, 10, 13, 16, 26, 29, 35\}, \{5, 7, 12, 16, 20, 32, 39\}, \{5, 9, 14, 18, 21, 29, 38\}, \{5, 10, 11, 17, 22, 30, 37\}, \{6, 7, 13, 18, 22, 33, 36\}, \{6, 8, 12, 17, 21, 34, 35\}, \{6, 10, 14, 15, 19, 27, 39\}\}.
\subsection{List of Files}\label{Sec:files}
Here is a list of all the files that we provide for the project, under the folder ``cases". The case number of each file corresponds directly to the subsection number under Section \ref{Sec:Result} so we won't give redundant reference in the table. For each program, the ``input file" name is already written in the code. Each program will directly print the result that is supposed to be the same as what is written in the ``output file". The run time approximation is rough and serves as an upper bound. We also provide a file \verb|testcases.sh| to automatically test that the output files we provided are correct. The run time of \verb|testcases.sh| is supposed to be the sum of run times listed below. Readers should refer to \verb|README.txt| for more details.

\

\begin{tabular}{|c|c|c|c|}
\hline
file name & input file & output file & run time \\
\hline
\verb|case1-1-phase1.cpp| & \verb|case1-phase0.txt| & \verb|case1-1-phase1.txt| & 1 min \\
\hline
\verb|case1-1-phase2.cpp| & \verb|case1-1-phase1.txt| & \verb|case1-1-phase2_SPPP.txt| & 1 min \\
\hline
\verb|case1-2-phase1.cpp| & \verb|case1-phase0.txt| & \verb|case1-2-phase1.txt| & 1 min \\
\hline
\verb|case1-2-phase2.cpp| & \verb|case1-2-phase1.txt| & \verb|case1-2-phase2_SPPP.txt| & 1 min \\
\hline
\verb|case2-phase1.cpp| & \verb|case2-phase0.txt| & \verb|case2-phase1_SPPP.txt| & 10 days \\
\hline
\verb|case3-phase1.cpp| & \verb|case3-phase0.txt| & \verb|case3-phase1.txt| & 10 hours \\
\hline
\verb|case3-phase2.cpp| & \verb|case3-phase1.txt| & \verb|case3-phase2_SPPP.txt| & 1 min \\
\hline
\verb|case4-phase1.cpp| & \verb|case4-phase0.txt| & \verb|case4-phase1.txt| & 1 hour \\
\hline
\verb|case4-phase2.cpp| & \verb|case4-phase1.txt| & \verb|case4-phase2.txt| & 2 days \\
\hline
\verb|case4-phase3.cpp| & \verb|case4-phase2.txt| & \verb|case4-phase3_SPPP.txt| & 2 min \\
\hline
\verb|case5-phase1.cpp| & \verb|case5-phase0.txt| & \verb|case5-phase1.txt| & 1 day \\
\hline
\verb|case5-phase2.cpp| & \verb|case5-phase1.txt| & \verb|case5-phase2.txt| & 1 min \\
\hline
\verb|case5-phase3.cpp| & \verb|case5-phase2.txt| & \verb|case5-phase3_SPPP.txt| & 2 days \\
\hline
\end{tabular}

\subsection*{Acknowledgements}
Thanks to Henry Cohn for supervising this project.
\bibliographystyle{plain}
\bibliography{ref}

\begin{thebibliography}{1}

\bibitem{assmus1970possibility}
E.~F. Assmus~Jr and H.~F. Mattson~Jr.
\newblock On the possibility of a projective plane of order 10.
\newblock {\em Algebraic Theory of Codes II}, 1970.
\newblock Air Force Cambridge Research Laboratories Report AFCRL-71-0013,
  Sylvania Electronic Systems, Needham Heights, Mass.

\bibitem{bose1938application}
Raj~Chandra Bose.
\newblock On the application of the properties of galois fields to the problem
  of construction of hyper-graeco-latin squares.
\newblock {\em Sankhy{\=a}: The Indian Journal of Statistics}, 3(4):323--338,
  1938.

\bibitem{bruck1949nonexistence}
Richard~H. Bruck and Herbert~J. Ryser.
\newblock The nonexistence of certain finite projective planes.
\newblock {\em Canad. J. Math}, 1:88--93, 1949.

\bibitem{hall1943projective}
Marshall Hall.
\newblock Projective planes.
\newblock {\em Transactions of the American Mathematical Society},
  54(2):229--277, 1943.

\bibitem{hering2007partial}
Christoph Hering and Andreas Krebs.
\newblock A partial plane of order 6 constructed from the icosahedron.
\newblock {\em Designs, Codes and Cryptography}, 44(1):287--292, 2007.

\bibitem{lam1989non}
Clement W.~H. Lam, T.~Thiel, and Stanley Swiercz.
\newblock The non-existence of finite projective planes of order 10.
\newblock {\em Canad. J. Math}, 41:1117--1123, 1989.

\bibitem{mccarthy1976approximations}
R.~C. McCarthy, R.~C. Mullin, P.~J. Schellenberg, R.~G. Stanton, and
  SA~Vanstone.
\newblock On approximations to a projective plane of order 6.
\newblock {\em Ars Combinatoria}, 2:169--189, 1976.

\bibitem{McKay201494}
Brendan~D. McKay and Adolfo Piperno.
\newblock Practical graph isomorphism, {II}.
\newblock {\em Journal of Symbolic Computation}, 60:94--112, 2014.

\bibitem{prince2009pure}
Alan~R. Prince.
\newblock Pure partial planes of order 6 with 25 lines.
\newblock {\em Designs, Codes and Cryptography}, 52(2):243--247, 2009.

\end{thebibliography}
\end{document}